\documentclass[11pt,twoside]{article}
\usepackage{a4}
\usepackage{amssymb,amsmath,amsthm,latexsym}
\usepackage{amsfonts}
\usepackage{amsfonts}
\usepackage{graphicx}
\usepackage{nicematrix}
\usepackage{multirow}

\newtheorem{theorem}{Theorem}[section]

\newtheorem{definition}[theorem]{Definition}

\newtheorem{lemma} [theorem]{Lemma}

\newtheorem{proposition}[theorem]{Proposition}
\newtheorem{remark}[theorem]{Remark}

\usepackage{caption}

\vspace{5cm}

\begin{document}
	
	\label{'ubf'}  
	\setcounter{page}{1}                                 

	\markboth {\hspace*{-9mm} \centerline{\footnotesize \sc
			\scriptsize { Generalized Splitting and element splitting operations on $p$-matroids}}
	}
	{ \centerline                           {\footnotesize \small  
			Sachin Gunjal$^1,$ Uday Jagadale$^2,$  Prashant Malavadkar$^3$                                               } \hspace*{-9mm}              
	}

	\begin{center}
		{ 
			{\Large \textbf { \sc Generalized Splitting and element splitting operations on $p$-matroids }
			}
			\\

			\medskip

			{\sc Sachin Gunjal$^1,$ Uday Jagadale$^2,$  Prashant Malavadkar$^3$  }\\
			{\footnotesize  $^{1,2,3}$Dr. Vishwanath Karad MIT World Peace University,
				Pune 411 038,	India.\\
			{\footnotesize e-mail: {\it 1.sachin.gunjal@mitwpu.edu.in, 2.uday.jagdale@mitwpu.edu.in 3.prashant.malavadkar@mitwpu.edu.in,  }}
		}}
	\end{center}

	\thispagestyle{empty}

	\hrulefill

	\begin{abstract}  
		{\footnotesize In this paper, we define generalized splitting  and element splitting operations on $p$-matroids. $p$-matroids are the matroids representable over $GF(p).$ The circuits and the bases of the new matroid are characterized in terms of circuits and bases of the original matroid, respectively. A class of $n$-connected $p$-matroids which gives n-connected $p$- matroids using the generalized splitting operation is also characterized. We also prove that connectivity of $p$-matroid is preserved under element splitting operation. Sufficient conditions to obtain Eulerian $p$-matroid from Eulerian $p$-matroid under splitting and element splitting operations are provided. 
			}
	\end{abstract}
	\hrulefill

	

	\section{Introduction}
 Two element splitting operations on binary matroids is introduced by Raghunathan et al. \cite{RSWS}, and the element splitting operations with respect to $n$-elements on binary matroids is introduced by Shikare \cite{shika}. Also Shikare extended the two elements splitting operation to $n$ element splitting operation in \cite{bingen}. Properties of splitting and element splitting operations are studied in \cite{AM, MDSCGEltSM, SSL, SABS, YWU}. Further,the splitting and the element splitting operation, with respect two elements, on the matroids representabe over $GF(p),$ called $p$-matroid, are discussed in \cite{mjg, elt}. \\
 
 \noindent In the present paper, we discuss the effect of the splitting and the element splitting operation with respect to $n$ elements on $p$ matroids. 
	 
 \noindent We consider only loopless and coloopless $p$-matroids in the present paper. For undefined terms and notions refer to \cite{Oxley}.
	
	\section{Generalized Splitting, and Generalized Element Splitting Operations}
	In this section, we define splitting and element splitting operations with respect to $n$-elements, henceforth called generalized splitting, and generalized element splitting operations, respectively.
	 
	\begin{definition}\label{D1}
		Let $M$ be a $p$-matroid on a set $E$ and let $T\subset E.$ Suppose $A$ is a matrix representation of $M$ over the prime field $GF(p).$ Let $A_T$ be the matrix obtained by adjoining an extra row with entries zero everywhere except in the column corresponding to the members of T where it takes the value $1.$ The vector matroid of matrix $A_T$ is denoted by $M_T.$ The transition from $M$ to $M_T$ is called splitting operation and the matroid $M_T$ is called the splitting matroid. Let $A'_T$ be the matrix obtained from $A_T$ by adjoining an extra column labeled $z$ with entries zero everywhere except in the last row where it takes the value $1.$ The vector matroid of matrix $A'_T$ is denoted by $M'_T.$ The transition from $M$ to $M'_T$ is called element splitting operation and the matroid $M'_T$ is called the element splitting matroid.
	\end{definition}

\begin{remark}
If $T\subset M$ is cocircuit of $M,$ then $M \equiv M_T.$ In such a case the splitting is called a trivial splitting; otherwise it is called non-trivial splitting.  
\end{remark}
 \noindent Let $C$ be a circuit of $M,$ and $C \cap T \neq \phi.$ We say $C$ is a \textbf{PT-circuit} of $M,$ if it is a also circuit of $M_T.$ And if $C$ is not a circuit of $M_T,$  we call it an \textbf{NPT- circuit} of $M.$ We use $\mathcal {C}_0$ to denote the collection of $PT$-circuits and the circuits containing no element of $T.$
 Consider subsets of $E$ of the type $C\cup I$ where $C=\{u_1,u_2,\ldots,u_l\}$ is an $NPT$-circuit of $M$ which is disjoint from an independent set $I=\{v_1,v_2,\ldots,v_k\}$ and $\ T\cap(C\cup I)\neq \phi$. We say $C\cup I$ is a \textbf{PT-dependent set} if it contains no member of $\mathcal {C}_0$ and there are non-zero constants $\alpha_1,\alpha _2,\ldots,\alpha_l$ and $\beta_1, \beta_2,\ldots,\beta_k$ such that $\sum_{i=1}^{l}\alpha_i u_i + \sum_{j=1}^{k}\beta_j v_j = 0 $ and $ \displaystyle \sum_{x \in T\cap(C\cup I)} coeff.(x)\equiv 0(mod \: p).$\\
 
	\begin{lemma} \label{dep}
		Let $M$ be a $p$-matroid and $C_1, C_2$ are disjoints $NPT$-circuits of $M.$ Then $C_1\cup C_2$ is a dependent set in $M_T.$ 
	\end{lemma}
\begin{proof}
	Let $C_1=\{x_1, x_2,\ldots, x_k\},$ $C_2=\{y_1, y_2,\ldots, y_l\},$ $|C_1\cap T|=k_1,$ and $|C_2 \cap T|=k_2.$ Without loss of generality assume $x_1, x_2, \ldots, x_{k_1} \in T,$ $y_1, y_2,\ldots, y_{k_2}\in T$ where $k_1 \leq k,$ $k_2\leq l.$ Since $C_1$ and $C_2$ are circuits, there exist non-zero scalars $\alpha_1, \alpha_2, \ldots, \alpha_k$ and $\beta_1, \beta_2, \ldots , \beta_l$ such that\[\alpha_1 x_1+\alpha_2 x_2+\ldots+\alpha_{k_1} x_{k_1}+\dots+\alpha_k x_k=0 (\mod p),\] \[\beta_1 y_1+\beta_2 y_2+\ldots+\beta_{k_2} x_{k_2} +\dots +\beta_l y_l=0(\mod p).\] Let $\alpha = \alpha_1+\alpha_2+\dots + \alpha_{k_1},$ and $\beta = \beta_1+\beta_2+\dots + \beta_{k_2}.$ Multipying the equation 2 by $\frac{-\alpha}{\beta},$ we get \[\beta'_1 y_1+\beta'_2 y_2+\ldots+\beta'_{k_2} x_{k_2} +\dots +\beta'_l y_l=0(\mod p).\] Note that all the coefficients $\beta'_i, i=1,2,\dots,l$ are non-zero scalars. Thus, we get non-zero scalars $\alpha_1,\alpha,\ldots, \alpha_k, \beta'_1,\beta'_2,\ldots,\beta'_l$ such that \[(\alpha_1 x_1+\alpha_2 x_2+\ldots+\alpha_{k_1} x_{k_1}+\dots+\alpha_k x_k)+(\beta'_1 y_1+\beta'_2 y_2+\ldots+\beta'_{k_2} x_{k_2} +\dots +\beta'_l y_l)=0\]  Further, \[(\alpha = \alpha_1+\alpha_2+\dots + \alpha_{k_1})+(\beta'_1+\beta'_2+\dots + \beta'_{k_2})=0,\] i.e. the sum of the coefficients of the elements in $(C_1 \cup C_2)\cap T$ is zero. Therefore $C_1 \cup C_2$ is dependent set of $M_T.$
\end{proof}
	\begin{remark}
		Let $C_1$ and $C_2$ be $NPT$-circuits of $M,$ and $I$ be an independent set of $M.$ Then $C_1\cup C_2\cup I$ cannot be a circuit of $M_T.$
	\end{remark}

\noindent The above lemma characterizes all the circuit of $M'_T$ containing $z.$ We denote the class of such circuits by $\mathcal {C}_z.$ Thus
\[ \mathcal {C}_z = \{ C\cup z: C~ is~ NPT~ circuit~ of~ M \} \]
\begin{lemma}\label{L2}
	Let $C$ be a circuit of $M'_T.$ Then $z \in C$ if and only if $C\setminus z$ is an $NPT$-circuit of $M.$
\end{lemma}
\begin{proof}
	Assume that $z \in C$ is a circuit of $M'_T.$ Let $C'= C\setminus e.$ Then $C'$ is a dependent set of $M.$ If $C'$ itself is a $PT$ circuit, or it contains a $PT$ circuit, then such a circuit is also a circuit of  $M'_T.$ Thus there exists a circuit contained in $C,$ a contradiction. Next assume that $C'$ contains an $NPT$ circuit $C_1$ of $M.$ Then $C_1$ is an independent set of $M'_T.$ Note that $\sum_{x\in C_1\cap T}$ coeff.$(x)$  $ \not\equiv 0(mod ~p).$ Let  $\sum_{x\in C_1\cap T}$ coeff.$(x)$  $=\alpha.$ Multiplying the column $z$ by $-\alpha,$ the set $C_1\cup z$ is a dependent set of $M'_T$ contained in $C,$ which is a contradiction. Therefore we conclude that $C$ must be and $NPT-$ circuit. \\
	Conversely, suppose $C'$ is an $NPT$-circuit of $M.$ Then as argued earlier the set $C'\cup z$ is a dependent set in $M'_T$. If $C'\cup z$ is not a circuit of $M'_T$ then it contains a circuit say $C_1$ of $M'_T.$ One of the following case will occur;\\
	\textbf{Case (i)} If $z\in C_1,$ then $C_1\setminus z$ is dependent set of $M$ contained in $C',$ which is not possible.\\
	\textbf{Case (ii)} If $z\notin C_1,$ then $C_1\subseteq C',$ a contradiction to the fact that $C'$ is an independent set of $M'_T.$
	Hence $C'\cup z$ is a circuit of $M'_T.$
\end{proof}

\section{Circuits, Bases, and Rank function of $M_T$  and $M'_T$}

\noindent Theorem \ref{L1} describes the circuits of the splitting matroids $M_T,$ and the element splitting matroids $M'_T$ in terms of the circuits of the original $p$-matroids $M.$
\begin{theorem}\label{L1}
	Let $M$ be a $p$-matroid on the ground set $E$ and $T \subset E.$ Then
	\begin{enumerate}
		\item  $\mathcal{C}(M_T)= \mathcal{C}_0\cup\mathcal {C}_1\cup \mathcal {C}_2$
		\item  $\mathcal{C}(M'_T)= \mathcal{C}_0\cup\mathcal {C}_1\cup \mathcal {C}_2\cup \mathcal{C}_z$
	\end{enumerate}
	 where
	\begin{description}
		\item $\mathcal {C}_0 = \{ C \in \cal C$$(M)~ :~ C$ is a $PT$-circuit or $C\cap T=\phi$ $\}$;
		\item $\mathcal {C}_1 =$  minimal ~elements ~of $\{ C \cup I : (C\cup I)$ is a $PT$-dependent set of $M$ \};
		\item $\mathcal {C}_2 =$  minimal ~elements ~of $\{ C_1 \cup C_2 : C_1, C_2~$ are  $NPT$ circuits, $C_1\cap C_2 = \phi$ and $C_1 \cup C_2$ contains no member of $\mathcal{C}_0,$ and  $\mathcal{C}_1$\};
		
		\end{description}
\end{theorem}
\begin{proof}
	In proving $1$ the inclusion $ \mathcal{C}_0\cup\mathcal {C}_1\cup \mathcal {C}_2\subseteq \mathcal{C}(M_T)$ follows from the Definition \ref{D1} and Lemma \ref{dep}. For other way inclusion, let $C$ be a circuit of $M_T.$ Then $C$ is a dependent set of $M.$ Following are the three possible cases for $C$ to be in $M:$\\
	Case-I: $C$ is a circuit in $M.$ Then $C \in \mathcal{C}_0.$\\ 
	Case-II: $C$ is not a circuit of $M.$  Then $C$ must contain a circuit $C_1$ of $M.$ Now consider $C\setminus C_1:$\\
	Subcase-I: The set $C\setminus C_1$ is dependent in $M.$ Then note that $(C\setminus C_1)\cap T \neq \phi.$ Otherwise $C\setminus C_1$ contains an $NPT$ circuit, say $C'_1,$ contained in the circuit $C$ of $M_T.$ which is impossible. Therefore $C\setminus C_1$ must contain an $NPT$ circuit,say $C_2.$ The set $S= C\setminus(C_1 \cup C_2)$ cannot be non empty in $M,$ otherwise the set $C=C_1 \cup C_2 \cup S$ contains two disjoint $NPT$-circuits $C_1,C_2.$ By the Lemma \ref{dep} $C_1 \cup C_2$ is dependent set of $M_T$ contained in the circuit $C$ of $M_T,$ a contradiction. In this case, $C \in \mathcal{C}_2$\\
	Subcase-II: The set $C\setminus C_1$ is independent in $M.$ Then $C=C_1\cup I,$ where $I=C\setminus C_1.$ The set $I \cap T \neq\phi,$ otherwise $C_1$ is an $NPT-$circuit. Therefore $C\in \mathcal{C}_2.$ \\
	In proving $2$ the inclusion $ \mathcal{C}(M_T)\cup \mathcal{C}_z \subseteq \mathcal{C}(M'_T) $follows from the Definition \ref{D1} and Lemma \ref{L2}. For the other inclusion let $C$ be the circuit of $M'_T.$ \\
	\textbf{Case (i)} Let $z\in C.$ Then by the Lemma \ref{L1}, $C\in \mathcal {C}_z.$\\
	\textbf{Case (ii)} If $z\notin C,$ then $C\in \mathcal {C}(M_T).$\\
	Therefore $\mathcal{C}(M'_T)\subseteq \mathcal{C}(M_T)\cup \mathcal{C}_z. $

\end{proof}

\noindent The collection of the bases of $M_T$ and $M'_T$ are denoted by $\mathcal{B}(M_T),$ and $\mathcal{B}(M'_T),$ respectively. We now provide the description of the members of $\mathcal{B}(M_T),$ and $\mathcal{B}(M'_T),$ in terms of the basis elements of $M.$
\begin{theorem}\label{k}
	Let $M$ be a $p$-matroid, $T \subseteq E,$ and $M$ contains an $NPT$-circuit. Then
	\begin{enumerate}
		\item $\mathcal{B}(M_T)=\mathcal{B}_1=\{B\cup x:B\in \mathcal{B}(M), x\notin B$ and $B\cup x$ contains neither $PT$-circuit nor $PT$-dependent set\}
		\item $\mathcal{B}(M'_T)=\mathcal{B}_1 \cup \mathcal{B}_z,$ where $\mathcal{B}_z = \{ B \cup z : B \in \mathcal{B}(M)\}.$
	\end{enumerate}
	
\end{theorem}
\begin{proof}
	Proof of $1$ follows by using similar arguments given in the proof of Theorem 2.12 of \cite{mjg}.\\
	In proving $2,$ it is easy to observe that $\mathcal{B}(M_T)\cup \mathcal{B}_z \subseteq \mathcal{B}(M'_T).$ Next assume that $B \in \mathcal{B}(M'_T).$ Then $rank (B)=rank (M)+1.$ Let $z \in B.$ Then $B'= B\setminus z$ is an independent set of $M_T$ of size $rank(M).$ If $B'$ is also independent set of $M,$ then $B \in \mathcal{B}_z.$ If $B'$ is  a dependent set of $M,$ then it must contain an $NPT-$ circuit $C.$ By Lemma \ref{L1}, $C\cup z$ is a circuit of $M'_T$ contained in the basis $B$ of $M'_T,$ a contradiction. If $z\notin B,$ then $B$ is an independent set of size $rank(M)+1.$ Therefore $B \in \mathcal{B}(M_T).$
\end{proof}
\noindent The rank functions of $M,$ $M_T,$ and $M'_T$ are denoted by $r,$ $r'$ and $r"$ respectively.
\begin{theorem}\label{s}
	Suppose $S\subseteq E(M).$ Then
	\begin{equation}
	\begin{split}
	r'(S) &= r(S) , \text { ~~~~~  if S contains no NPT-circuit of M; and}\\
	&= r(S)+1,\text {~   if S contains an NP-circuit of M.}
	\\
	r''(S \cup z) & = r(S) + 1. 
	\end{split}
	\end{equation}
\end{theorem}

\begin{proof}
	Let $B$ and  $B'$ be the bases of $M|_S$ and $M_T|_S,$ respectively, then $r(S)=|B|$ and $r'(S)=|B'|.$ If $S$ contains no $NPT$-circuit, then $|B'|=|B|.$ Therefore $r(S)=r'(S).$ If $S$ contains an $NPT$-circuit, then by the Theorem \ref{k}, $B'=B \cup  e$ for some $B \in \mathcal{B}(M|_S),$$e\in (S\setminus B)$ and $B\cup e$ contains a unique $NPT$-circuit. Therefore $r'(S)=r'(M_T|_S)=|B'|=|B\cup e|=|B|+1=r(S)+1.$\\
	The equality $r''(S \cup z) = r(S) + 1 $follows from the Definition \ref{D1}.
\end{proof}

\section{Effect of splitting and element splitting operations on connected and Eulerian $p$-matroids}
Higher connectivity of graphs and matroids is well explored in \cite{SlaterC4CG, BixbyT3C, TutteLM, TutteT3CG, TutteCM}. In this section we study the effect of splitting and element splitting operations on connected and Eulerian $p$-matroids.  In the following lemma, we provide a sufficient condition to obtain a connected $p$-matroid from a connected $p$-matroid using splitting operation.
\begin{lemma}\label{con}
	Let $M$ be a connected $p$-matroid and $T\subset E(M)$. If for every proper subset $X$ of $E(M)$ with $|X|\geq 1$, either $X$ or $Y = E(M)\setminus X$ contains an $NPT$-circuit of $M,$ then $M_T$ is connected.
\end{lemma}

\begin{proof}
	Note that, $M$ is a connected $p$-matroid,which implies that $M$ has no $1$-separation. On the contrary, assume $M_T$ is not connected. That is, $M_T$ has $1$-separation,say $(X,Y).$ Therefore \[r'(X) + r'(Y) -r'(M_T) < 1.\] If $X$ and $Y$ both contains $np$-circuits then, by lemma \ref{s},  we have\[r(X)+1+r(Y)+1-r(M)-1 = r(X)+r(Y)-r(M) + 1 < 1.\] \noindent Thus, we get a contradiction to the fact $r(X)+r(Y)-r(M)\geq 0. $ Further, if only one of $X$ or $Y,$ say $X,$ contains an $np$- circuit, then we have\[r'(X) + r'(Y) -r'(M_T) = r(X)+1+r(Y)-r(M)-1 < 1.\] That is  \[r(X)+r(Y)-r(M) < 1.\] Thus $(X,Y)$ gives a $1$-separation of $M$ which is not possible.
\end{proof}
\noindent  Theorem \ref{nconn} characterizes a class of $n$-connected $p$-matroids whose $n$-connectivity is preserved under splitting operation.
\begin{theorem} \label{nconn}
Let $M$ be an $n$-connected and vertically $(n + 1)$-connected binary
matroid, $n\geq 2,~ |E(M)|\geq 2(n-1)$ and girth of $M$ is at least $n+1.$ Let $X\subset E(M)$ with $|X|\geq n.$ Then $M_X$ is $n$- connected if and only if for any $(n-1)$ element subset $S$ of $E(M)$ there is an $NPT$-circuit $C$ of $M$ such that $S\cap C =\phi.$
\end{theorem}
\begin{proof}
This can be proved using similar arguments used in the Theorem 2.3 of \cite{conn} by replacing $OX$-circuit by $NPT$-circuit.
\end{proof}

 \noindent The next theorem states that the non-trivial element-splitting operation preserves the connectivity of $p$-matroid.
\begin{theorem}
Let $M$ be a connected $p$-matroid on ground set $E.$ Then $M'_T$ is a
connected $p$-matroid on ground set $E\cup \{z\}$ if and only if $M_T$ is the splitting matroid obtained by applying non-trivial splitting operation on $M.$
\end{theorem}
\begin{proof}
	Proof is similar to the proof of Theorem 3.1 of \cite{elt}. In this proof, splitting is of two elements $a,b$ but same arguments holds true for general splitting matroid $M_T,$ $|T|\geq 2.$ 
\end{proof}
\noindent Fleischner \cite{FS} used the operation of splitting away a pair of edges from a vertex of degree at least three on graphs to characterize Eulerian graphs. In \cite{DKWag, Welsh} Eulerian matroids are studied. Let $M$ be an Eulerian matroid on ground set $E.$ Then there are disjoint circuits $C_1,$$C_2,$ $\ldots$,$C_{2l-1},$$C_{2l},$$\ldots$,$C_k;2l\leq k,$ of $M$ such that
$E= C_1 \cup C_2 \cup \ldots \cup C_{2l-1} \cup C_{2l}\cup \dots \cup C_k.$
\noindent Let $T \subset E$ and $M_T$ be the splitting matroid. We say that the collection \~{C} $=\{C_1,C_2,\ldots, C_{2l-1},C_{2l},\ldots, C_k\}$ of  disjoint circuits is a \textbf{PT-decomposition}  of $M,$ if for $1\leq j \leq l,$ $T \cap C_{2j-1} \neq \phi,$ $T\cap C_{2j}\neq \phi$  and $(C_{2j-1} \cup C_{2j}) \in \mathcal{C}_2,$ or $C_{2j-1}, C_{2j}$ are $PT$- circuits;  and $ C_j \in \mathcal{C}_0$ for all $j \in \{2l+1,2l+2,\ldots,k\}$

\begin{proposition}\label{e1}
	Let $M$ be a $p$-matroid on ground set $E$ and $T \subset E$. If $M$ is an Eulerian matroid having a $PT$-decomposition, then $M_T$ is Eulerian. 
\end{proposition}

\begin{proof}
	Let  \~{C} $=\{C_1,C_2,\ldots, C_{2l-1},C_{2l},\ldots, C_k\}$ be a $PT$-decomposition of $M.$ Then \[E= C_1 \cup C_2 \cup \ldots \cup C_{2l-1} \cup C_{2l}\cup \dots \cup C_k, \]
	
	\noindent  $T \cap C_{2j-1} \neq \phi,$ $T\cap C_{2j}\neq \phi$  for $1\leq j \leq l;$ and  $ C_j \in \mathcal{C}_0$ for all $j \in \{2l+1,2l+2,\ldots,k\}$.Let $C'_j =(C_{2j-1} \cup C_{2j}); j \in \{1,2,\ldots, l\}.$ Then by the 
	 of the $PT$-decomposition $C'_j \in \mathcal{C}_2.$ The collection $\{C'_1,\dots, C'_l, C_{2l+1},\dots, C_k\}$ is a circuit decomposition of $M'_T.$ Therefore, $M_T$ is an Eulerian matroid.\\

\end{proof}
\begin{proposition}
Let $M$ be Eulerian $p$-matroid and $T\subset E$. If circuit decomposition of $M$ contains exactly one $NPT$-circuit with respect to $T$ split then $M'_T$ is Eulerian p -matroid.	
\end{proposition}

\bibliographystyle{amsplain}

\end{document}